\documentclass[11pt]{amsart}
\usepackage{amssymb,amsmath,amsfonts,amscd,euscript}

\newcommand{\nc}{\newcommand}

\numberwithin{equation}{section}
\newtheorem{thm}{Theorem}[section]
\newtheorem{prop}[thm]{Proposition}
\newtheorem{lem}[thm]{Lemma}

\theoremstyle{remark}
\newtheorem{rem}[thm]{Remark}

\nc{\gl}{\mathfrak{gl}}
\nc{\GL}{\mathfrak{GL}}
\nc{\g}{\mathfrak{g}}
\nc{\gh}{\widehat\g}
\nc{\h}{\mathfrak{h}}
\nc{\la}{\lambda}
\nc{\C}{\mathbb C }
\nc{\Z}{\mathbb Z }
\nc{\N}{\mathbb N }
\nc{\R}{\mathbb R }
\nc{\Q}{\mathbb Q }
\nc{\al}{\alpha }

\nc{\oM}{{\overline{\mathcal{M}}}}
\nc{\ta}{\theta}
\nc{\ve}{\varepsilon}
\nc{\ch}{{\mathop {\rm ch}}}
\nc{\Tr}{{\mathop {\rm Tr}\,}}
\nc{\Id}{{\mathop {\rm Id}}}
\nc{\ad}{{\mathop {\rm ad}}}
\nc{\bra}{\langle}
\nc{\ket}{\rangle}
\nc{\x}{{\bf x}}
\nc{\pa}{\partial}
\nc{\ld}{\ldots}
\nc{\cd}{\cdots}
\nc{\hk}{\hookrightarrow}
\nc{\T}{\otimes}
\nc{\be}{\begin{equation}}
\nc{\en}{\end{equation}}
\nc{\ben}{\begin{equation*}}
\nc{\enn}{\end{equation*}}
\nc{\gr}{\mathrm{gr}}
\nc{\ov}{\overline}

\nc{\cO}{\mathcal O}
\nc{\Cl}{\mathcal{C}l}
\def\un{{1\!\! 1}}

\begin{document}

\title[Givental symmetries and KP tau-functions]
{Givental symmetries of Frobenius manifolds and multi-component KP tau-functions}

\author{Evgeny Feigin}
\address{Evgeny Feigin:\newline
Tamm Theory Division,
Lebedev Physics Institute,\newline
Leninisky prospect, 53,
119991, Moscow, Russia
}
\email{evgfeig@gmail.com}

\author{Johan van de Leur}
\address{Johan van de Leur: \newline Mathematical Institute,
University of Utrecht,\newline
P.~O.~Box 80010, 3508 TA Utrecht,
The Netherlands}
\email{J.W.vandeLeur@uu.nl}

\author{Sergey Shadrin}

\address{Sergey Shadrin:\newline
Korteweg-de~Vries Institute for Mathematics,
University of Amsterdam,\newline
P.~O.~Box 94248, 1090 GE Amsterdam,
The Netherlands\newline
\indent and\newline
Department of Mathematics, Institute of System Research,\newline
Nakhimovsky prospekt 36-1, Moscow 117218, Russia}
\email{s.shadrin@uva.nl, shadrin@mccme.ru}

\begin{abstract}
We establish a link between two different constructions
of the action of the twisted loop group on the space of Frobenius structures.
The first construction (due to Givental) describes the action of the
twisted loop group on the partition functions of formal Gromov-Witten theories.
The explicit formulas for the corresponding tangent action were computed by Y.-P.~Lee.
The second construction (due to van de Leur) describes the action of the same
group on the space of Frobenius structures via the multi-component KP hierarchies.
Our main theorem states that  the genus zero restriction of the Y.-P.~Lee formulas
coincides with the tangent  van de Leur action.
\end{abstract}

\maketitle

\tableofcontents

\section*{Introduction}
The formal Gromov-Witten theory was developed by Givental in~\cite{G1} as a very convenient
formalism to encode higher genera of Gromov-Witten potential in terms of genus $0$ data
associated to the underlying Frobenius manifold. In particular examples of Gromov-Witten
theories of Fano varieties, Givental's formulas reflect the structure of computations made
via localization technique~\cite{G4}. It was proved recently by Teleman~\cite{T} that in
the semisimple
case the partition function of an arbitrary cohomological field theory coincides with some
formal Gromov-Witten potential in the sense of Givental.

The core of Givental's theory is the action of the twisted loop
group of $GL(n)$ on the space of tame partition functions.
Equivalently, a half of the action of Givental's group can be
described as an action on cohomological field theories; there are
formulas for this action given explicitly in terms of tautological
classes on $\oM_{g,n}$~\cite{Kaz,T,KKP}. The key computation here was made
by Y.-P.~Lee, who provided explicit formulas for the Lie algebra
action of the Givental group~\cite{Lee1}, \cite{Lee2} (in a
different way it was done also by Chen, Kontsevich, and Schwarz,
see~\cite{CKS}).

In general, the Givental theory has appeared to be one of the principle tools
in the study of a variety of questions arising in relations to mirror symmetry,
Frobenius manifolds, usual and orbifold Gromov-Witten theory, geometry of $r$-spin
structure, and, more generally, quantum singularity theory, see
e.~g.~\cite{CZ,CCIT,CR,FJR,T} and many other recent preprints.
In this paper, however, we restrict the discussion to the Givental group action
(or, rather, Y.-P.~Lee's Lie algebra action) on non-homogeneous formal
Frobenius manifolds, that is, the semi-simple solutions of the WDVV equation,
not necessary homogeneous, but with flat unit.

A bit earlier than Givental, the second author, van de Leur, has observed that
in some special situation the equations of the multi-component KP hierarchy,
written for the wave functions, specialize to the so-called Darboux-Egoroff equations,
which is an equivalent way to write WDVV equations in canonical coordinates.
The case when it works is precisely the case when one considers the twisted
loop group of $GL(n)$ orbit of the vacuum vector rather than a generic
element of the semi-infinite Grassmannian. Using Dubrovin's formulas, one
can associate then a non-homogeneous formal Frobenius manifold to an element of
the twisted loop group of $GL(n)$.

So, we have two different actions of the same group on the same
spaces. It is a natural question to compare them. This question is
especially important, since it allows to link better two
completely different points of view on the theory of Frobenius
manifolds. Indeed, in applications to enumerative geometry and
string theory, Frobenius manifolds appear more naturally in flat
coordinates. Meanwhile, from the point of view of integrable
hierarchies, Frobenius manifolds appear more naturally in canonical
coordinates. The interplay between these two different points of view
was enormously developed by Dubrovin~\cite{D1,D}, and there is a huge
outcome of translating problems and arguments from one language to
another, see e.g.~\cite{DZ1,DZ2}. However, direct translation from one language to another
is usually extremely difficult. So, an interpretation in both
approaches of the same group action is going to be very helpful,
especially since the action is transitive on semi-simple Frobenius
manifolds.

The main result of the paper is that we identify the Lie algebra actions
coming from Givental's and van de Leur's approaches, that is,
we derive explicitly Y.-P.~Lee's formulas from equations of multi-component
KP hierarchy. While the semi-simple orbits are known in advance to
be the same in both approaches, the group action appears to be different,
but in a completely controlled way. That is, we construct an explicit isomorphism
of the tangent actions of Lie algebras. Let us also
mention here that the Y.-P.~Lee formulas for the Lie algebra action for solutions
of the WDVV equation written in canonical coordinates in terms
of wave functions of multi-component KP hierarchy appears to be very simple and elegant.

\subsection*{The content of the paper}

Let us first describe the main theorem that we prove in this paper. All necessary definitions are given in the main text.

Givental group action can be considered as an action of the elements $A(\zeta)\in \overline{GL}_\infty^{(\pm)}$,
$A^t(-\zeta)A(\zeta)=\Id$, on the space formal Frobenius structures in genus $0$ with flat unit, that is,
roughly speaking, on the space of formal solutions $F(t^1,\dots,t^n)$ of the WDVV equation.
The Lie derivatives of the Givental action on $F$ can be written in terms of so-called deformed
flat coordinates, or, in other language, in terms of generating functions for
correlators with descendants at one point. We denote them by $\theta_\mu^{(d)}$.

The construction of van de Leur associates to any group element
$A(t)\in \overline{GL}_\infty^{(\pm)}$, $A^t(-t)A(t)=\Id$, some Frobenius structure $F$
and a collection of $\theta_\mu^{(d)}$, both in canonical coordinates.
The formulas for $F$, $\theta_\mu^{(d)}$, and flat coordinates are given
in terms of a bit modified wave functions of multi-component KP hierarchy, so-called $\Psi$-function, $\Psi=\Psi(A)$.

For any $A(t)\in \overline{GL}_\infty^{(\pm)}$ and $a(t)\in \bar\gl_\infty^{\pm}$, we compute an
explicit expression for
\begin{equation}\label{eq:eps-der}
\left. \frac{\partial}{\partial\varepsilon} \Psi(A\exp(\varepsilon a)) \right |_{\varepsilon=0}.
\end{equation}
that can be written down in terms of $a$ and $\Psi(A)$. This allows us also to express a
first order $\varepsilon$-deformation of $F$ in terms of $\Psi$-functions.
When we rewrite these formulas in terms of $\theta_\mu^{(d)}$ in
flat coordinates, we obtain exactly the reduction to genus $0$
of the formulas of Y.-P.~Lee for the tangent Givental action.

The paper is organized in the following way.

In Section $1$ we review main definitions and constructions of semi-infinite wedge
space and multli-component bosonization.

In Section $2$ we compute the restriction of the Y.-P.~Lee vector fields
to the space of generating functions of correlators in genus $0$ with no descendants,
that is, the expressions of the tangent Givental action on $F$ in terms of $\theta_\mu^{(d)}$.
We also rewrite these formulas in terms of KP modified wave functions $\Psi$.

In Section $3$ we compute explicitly the derivatives of the multi-component KP $\Psi$-functions
with respect to the twisted loop Lie algebra, i.~e., we compute explicitly~\eqref{eq:eps-der}.

Finally, in Section $4$, we compare the formulas from Sections $2$ and $3$ and
prove our main theorem.

Though we attempt to make our paper self-contained, some background on both approached to
the group action on Frobenius manifolds might be helpful. In addition to the basic references
given in introduction and throughout the text, we refer the reader to~\cite{C,F,vdLM},
where some further aspects of both approaches in genus $0$ are discussed in detail.

\section{Semi-infinite forms and multi-component bosonization}
In this section we recall the construction of multi-component
bosonization following \cite{KvdL,KvdL2} (see also \cite{DJKM}, \cite{JM}).
\subsection{Semi-infinite forms}
\label{semi}
Let $V$ be an infinite-dimensional vector space with a basis
$v_i$, $i\in\Z+\frac{1}{2}$ and
$V^*=\bigoplus_{i\in\Z+\frac{1}{2}} (\C v_i)^*$ be its
restricted dual. Let $F=\bigwedge^{\frac{\infty}{2}} V$ be the space
of semi-infinite forms with a basis given by semi-infinite monomials of the
form
\begin{equation}
v_{i_1}\wedge v_{i_2}\wedge\dots, \quad i_1>  i_2> \dots,
\end{equation}
and $i_{l+1}=i_l-1$  for $l$ big enough.
For $m\in\Z$ we set $|m\ket=v_{m-\frac{1}{2}}\wedge v_{m-\frac{3}{2}}\wedge\dots$.
For $v\in V$ and $\xi\in V^*$ the wedging
and contracting operators $\psi_v$ and $\psi_\xi$ are defined as follows:
\begin{gather*}
\psi_v (v_{i_1}\wedge v_{i_2}\wedge \dots) = v\wedge v_{i_1}\wedge v_{i_2}\wedge\dots,\\
\psi_\xi (v_{i_1}\wedge v_{i_2} \wedge\dots)=\sum_{l=1}^\infty (-1)^{l+1} \xi(v_{i_l})
v_{i_1}\wedge\dots\wedge v_{i_{l-1}}\wedge v_{i_{l+1}}\wedge\dots.
\end{gather*}
Note that the vector $v\wedge v_{i_1}\wedge v_{i_2}\wedge \dots$
can be rewritten in terms of the basis vectors using bi-linearity
and  skew-symmetry of the wedge product. We set
\begin{equation}
\psi^+_i=\psi_{v_{-i}},\quad \psi^-_i=\psi_{v_i^*},
\end{equation}
where $v_i^*\in V^*$ are defined by $v_i^*(v_j)=\delta_{i,j}$.
The operators $\psi_i^+, \psi^-_j$   satisfy the relations
\begin{equation}
\psi_i^\lambda\psi_j^\mu+\psi_j^\mu\psi_i^\lambda=\delta_{\lambda,-\mu}\delta_{i,-j}
\end{equation}
for all $i,j\in\Z + \frac{1}{2}$ and $\lambda,\mu=+,-$.
Thus $\psi^+_i$ and $\psi^-_j$ generate a Clifford algebra, which we denote by $\Cl$ .
We note that $F$ is an irreducible $\Cl$ module. For instance, $F=\Cl\cdot |0\ket$ with
the relations $\psi^\pm_j |0\ket=0$ for $j>0$.
Using the action of $\Cl$ we define
the energy and charge gradings on $F$. Namely set
\begin{align}
\text{charge} (|0\ket) & =0,
& \text{energy}(|0\ket) &=0,  \\
 \text{charge} (\psi^\pm_i) & =\pm 1,
&\text{energy} (\psi^\pm_i) & =-i.\notag
\end{align}
Let
\begin{align*}
F & =\bigoplus_{m\in\Z} F^{(m)},
& F^{(m)} & =\bigoplus_{d=\frac{m^2}2}^\infty F^{(m)}_d
\end{align*}
be charge and energy decompositions.
For instance,
$|m\ket\in F^{(m)}_{\frac{m^2}2}$.
For later use we introduce the left vacuum $\langle 0|\in F^*$ such that the vacuum expectation value
$\langle 0||0\ket=1$ and $\langle 0| v=0$ if energy of $v$ is positive.
We note that the operator $\langle 0|\psi^\pm_j$ vanishes for $j<0$.

We can also define representations of the infinite complex matrix
group $GL_\infty$ and its Lie algebra $\gl_\infty$ on $F$. Namely,
$GL_\infty$ is the group of invertible matrices
$A=(A_{ij})_{{i,j\in\Z +\frac{1}{2}}}$ such that all but a finite
number of $A_{ij}-\delta_{ij}$ are $0$. Then
\begin{equation}
\label{Lg}
A(v_{i_1}\wedge v_{i_2}\wedge  \dots)=Av_{i_1}\wedge Av_{i_2}\wedge\dots.
\end{equation}
The algebra  $\gl_\infty$ is the Lie algebra of all matrices  $a=(a_{ij})$ with a finite number
of non-vanishing entries. The action on $F$ is given by the formula
\begin{equation}
\label{La}
a(v_{i_1}\wedge v_{i_2}\wedge\dots)=\sum_{s=1}^\infty v_{i_1}\wedge\dots v_{i_{s-1}}\wedge av_{i_s}\wedge
v_{i_{s+1}}\wedge\dots.
\end{equation}
For any $A\in GL_\infty$ we can attach the element $\tau\in F$ which is given by $\tau=A|0\ket$.
We note that $A|0\ket\in F^{(0)}$ for any $A$.
The following proposition is standard:

\begin{prop}
A nonzero element $\tau$ of $F^{(0)}$ is equal to $A|0\ket$ for some $A\in GL_\infty$
if and only if the following equation holds in $F\T F$:
\begin{equation}
\sum_{i\in\Z+\frac{1}{2}} \psi^+_i \tau\T \psi^-_i \tau=0.
\end{equation}
\end{prop}

In what follows we will need a slight modification of the
construction above. Namely, let $\ov{GL}_\infty^+$ and
$\ov{GL}_\infty^-$ be the groups of invertible lower- and
upper-triangular matrices with units along the main diagonal.
Also, let $\bar \gl_\infty^+$ ($\bar \gl_\infty^-$) be the Lie
algebra of all strictly lower-(upper)triangular matrices. So in
particular,
\begin{equation}
a(v_i)=\sum_{j<i} a_{ji} v_j \text { for } a\in\bar\gl_\infty^-,\quad
a(v_i)=\sum_{j>i} a_{ji} v_j \text { for } a\in\bar\gl_\infty^+
\end{equation}
Note that we do not assume that the number of non-trivial entries of matrices
from the groups and algebras above is finite.

There are no problems with the action of the group
$\ov{GL}_\infty^-$ and the algebra $\bar\gl^-_\infty$. Namely, the
formulas \eqref{Lg} and \eqref{La} still work. In order to define
the action of the group $\ov{GL}_\infty^+$ and the algebra
$\bar\gl^+_\infty$ we need to complete the space $F$ with respect
to the energy grading. Namely, we consider the space
\begin{equation}
\bar F=\bigoplus_{m\in\Z} \bar F^{(m)},\quad
\bar F^{(m)}=\prod_{d=\frac{m^2}2}^\infty F^{(m)}_d
\end{equation}
So an element of $\bar F^{(m)}$ has a component with any possible energy in the
charge $m$ subspace $F^{(m)}$.
\begin{lem}
The formulas \eqref{Lg} and \eqref{La} define the action of
the group $\ov{GL}_\infty^+$ and
the algebra $\bar\gl^+_\infty$ on $\bar F$.
\end{lem}

\subsection{The $n$-component bosonization.}
From now on we fix $n\in\N$.
It will be convenient for us to relabel the vectors $v_i$ in the following
way: for $1\le j\le n$ and $k\in\Z+\frac{1}{2}$ we define
\[
v^{(j)}_k=v_{n(k-\frac{1}{2})+j-\frac{1}{2}}
\]
and the corresponding operators $\psi^{\pm (j)}_k$.
We note that
$$\psi^{\pm (j)}_{k}|0\rangle = 0\ \text{for}\ k > 0.$$
The relabelling above allows to introduce partial charges
\[
\text{charge}_j,\quad 1\le j\le n
\]
 in the following way:
$$\text{charge}_{j}(\psi^{\pm (i)}_{k})=\pm \delta_{ij},\
\text{charge}_{j}(|0\rangle) =0.
$$
Let
$$
F=\bigoplus_{k_1,\dots,k_n\in\Z} F^{(k_1,\dots,k_n)}
$$
be the decomposition with respect to the partial charges.
In particular, $F^{(m)}=\sum_{k_1+\dots+k_n=m} F^{(k_1,\dots,k_n)}.$
Similarly, the set of vectors $|m\ket$ is extended to the set
$|k_1,\dots, k_n\ket$, $k_i\in\Z$. In order to define these vectors
we need some additional operators
$Q_{i},\ i = 1,\ldots ,n$ on $F$.  These operators are uniquely defined by
the following conditions:
$$
Q_{i}|0\rangle = \psi^{+(i)}_{-\frac{1}{2}} |0\rangle ,\ Q_{i}\psi^{\pm
(j)}_{k} = (-1)^{\delta_{ij}+1} \psi^{\pm (j)}_{k\mp \delta_{ij}}Q_{i}.
$$
They satisfy the following commutation relations:
$$
Q_{i}Q_{j} = -Q_{j}Q_{i}\ \text{if}\ i \neq j.
$$
Note that $\langle 0|Q_{i}^{-1}=\langle 0|\psi^{-(i)}_{\frac{1}{2}}$.
We set
$$
|k_{1},\ldots ,k_{n}\rangle = Q^{k_{1}}_{1} \ldots
Q^{k_{n}}_{n}|0\rangle.
$$
For example,
$$
|\underbrace{0,\dots,0}_{i-1},1,0,\dots,0\ket=v^{(i)}_{\frac{1}{2}}\wedge
v^{(1)}_{-\frac{1}{2}}\wedge \dots v^{(n)}_{-\frac{1}{2}}\wedge v^{(1)}_{-\frac{3}{2}}\wedge\dots
$$
Obviously, $|k_1,\dots,k_n\ket\in F^{(k_1,\dots,k_n)}$.

It turned out that each space $F^{(k_1,\dots,k_n)}$ can be identified
with the space of polynomials in variables $x^{(i)}_k$, $1\le i\le n$, $n\in\Z$.
More precisely, each $F^{(k_1,\dots,k_n)}$ carries a structure of
a Fock module of the Heisenberg algebra based on $n$-dimensional space.
This identification is called the boson-fermion correspondence.
We describe it in details below.

Let us introduce the  fermionic fields
$$\psi^{\pm (j)}(z) = \sum_{k \in \Z+\frac{1}{2}}
\psi^{\pm (j)}_{k} z^{-k-\frac{1}{2}},
$$
which are formal generating functions of generators of the Clifford algebra.
Next we introduce bosonic fields $(1 \leq i,j \leq n)$:
$$
\alpha^{(ij)}(z) \equiv \sum_{k \in \Z} \alpha^{(ij)}_{k} z^{-k-1}
= :\psi^{+(i)}(z) \psi^{-(j)}(z):,
$$
where $:\ :$ stands for the {\it normal ordered product} defined in
the usual way $(\lambda ,\mu = +$ or $-$):
$$
:\psi^{\lambda (i)}_{k} \psi^{\mu (j)}_{\ell}: =
\begin{cases} \psi^{
\lambda (i)}_{k}
\psi^{\mu (j)}_{\ell}\ &\text{if}\ \ell > 0 \\
-\psi^{\mu (j)}_{\ell} \psi^{\lambda (i)}_{k} &\text{if}\ \ell < 0.
\end{cases}
$$
One can check that the operators
$\alpha^{(ij)}_{k}$ satisfy the commutation relations of the affine
algebra $\widehat{\gl}_{n}$ of the level $1$, i.e.:
$$[\alpha^{(ij)}_{p},\alpha^{(k\ell)}_{q}] =
\delta_{jk}\alpha^{(i\ell )}_{p+q}
- \delta_{i\ell} \alpha^{(kj)}_{p+q} + p\delta_{i\ell}
\delta_{jk}\delta_{p,-q},$$
and that
$$
\alpha^{(ij)}_{k}|m\rangle = 0 \ \text{if}\ k > 0 \ \text{or}\ k = 0\
\text{and}\ i < j.
$$
The operators $\alpha^{(i)}_{k} \equiv \alpha^{(ii)}_{k}$
satisfy the canonical commutation relation of the oscillator algebra
%which we
%denote by ${\frak a}$:
%
$$[\alpha^{(i)}_{k},\alpha^{(j)}_{\ell}] =
k\delta_{ij}\delta_{k,-\ell},$$
and one has
\begin{gather*}
\alpha^{(i)}_{\ell}|k_1,\dots,k_n\rangle = 0 \quad \text{for}\ \ell > 0, \\
\alpha^{(i)}_{0}|k_1,\dots,k_n\rangle =k_i|k_1,\dots,k_n\rangle,\quad i=1,\dots, n,
\end{gather*}
and $\langle 0|\alpha^{(i)}_{\ell}=0$ for $\ell\le 0$.
We also note that
$$
[\alpha^{(i)}_k,Q_j]=\delta_{ij}\delta_{k0}Q_j.
$$

One has the following vertex operator expression for  $\psi^{\pm
(i)}(z)$.
Given any sequence $s=(s_1,s_2,\dots)$, define
\[
\Gamma^{(j)}_\pm( s)=\exp \left(\sum_{k=1}^\infty s_k \alpha^{(j)}_{\pm
k}\right),
\]
then
\begin{thm}
\label{theorem}
$$
\begin{aligned}
\psi^{\pm (i)}(z) &= Q^{\pm 1}_{i}z^{\pm \alpha^{(i)}_{0}} \exp
(\mp \sum_{k < 0} \frac{1}{k} \alpha^{(i)}_{k}z^{-k})\exp(\mp
\sum_{k > 0} \frac{1}{k} \alpha^{(i)}_{k} z^{-k})\\[3mm]
&=Q^{\pm 1}_{i}z^{\pm \alpha^{(i)}_{0}}\Gamma_{-}^{(i)}(\pm [z])\Gamma_{+}^{(i)}(\mp
[z^{-1}]),
\end{aligned}
 $$
 where $[z]=(z,\frac{z^2}{2}, \frac{z^3}{3},\ldots)$.
\end{thm}

We note that
\begin{equation}\label{Gamma_vac}
\Gamma_+^{(j)}(s) \,|0\rangle
= |0\rangle  \,,\qquad
\langle 0| \, \Gamma_-^{(j)}(s)
= \langle 0|  \,.
\end{equation}
Let $\gamma(s,s')=e^{\sum n\, s_n s'_n}$.
For example,
\begin{equation}\label{e109}
\gamma(t,[z])= \exp\left(\sum_{n\ge 1} t_n\, z^n\right) \, .
\end{equation}
In what follows we will need certain properties of the functions $\Gamma_\pm^{(j)}$.
We collect them in the following proposition
\begin{prop}
We have the equalities:
\begin{equation}\label{Gamma_Gamma}
\Gamma_+^{(j)}(s)\, \Gamma_-^{(k)}(s') = \gamma(s,s')^{\delta_{jk}}\,
\Gamma_-^{(k)}(s')\, \Gamma_+^{(j)}(s) \, ,
\end{equation}
and
\begin{align}
\Gamma_\pm^{(j)}(s) \, \psi^{+(k)}(z)  &= \gamma(s,[z^{\pm 1}])^{\delta_{jk}} \,
\psi^{+(k)}(z) \, \Gamma_\pm^{(j)}(s) \, , \notag  \\
\Gamma_\pm^{(j)}(s) \, \psi^{-(k)}(z)  &= \gamma(s,-[z^{\pm 1}])^{\delta_{jk}} \,
\psi^{-(k)}(z) \, \Gamma_\pm^{(j)} (s)\, . \label{comgam}
\end{align}
\end{prop}

To obtain the multicomponent KP hierarchy we give another description of $F$. For more information see \cite{KvdL,KvdL2}.
Let $\C[x]$ be the space of polynomials in
variables $x = \{ x^{(i)}_{k}\},\ k = 1,2,\ldots ,\ i = 1,2,\ldots
,n$.  Let $L$ be a lattice with a basis $\delta_{1},\ldots
,\delta_{n}$ over $\Z$ and the symmetric bilinear form
$(\delta_{i}|\delta_{j}) = \delta_{ij}$, where $\delta_{ij}$ is
the Kronecker symbol.  Let
$$
\varepsilon_{ij} =
\begin{cases} -1 &\text{if $i > j$} \\
1 &\text{if $i \leq j$.}
\end{cases}$$
Define a bi-multiplicative function $\varepsilon :\ L \times L \to
\{\pm 1 \}$ by letting
$$
\varepsilon (\delta_{i}, \delta_{j}) = \varepsilon_{ij}.
$$

%  Let $\delta = \delta_{1} + \ldots + \delta_{n}$,
%  $ M = \{ \gamma \in L|\ (\delta | \gamma ) = 0\}$, $\Delta = \{ \alpha_{ij} :=\delta_{i}-\delta_{j}| i,j = 1,\ldots ,n,\ i \neq j \}$.
%  Of course $M$ is the root lattice of $s\ell_{n}(\C)$,
% the set $\Delta$ being the root system.

Consider the vector space $\C[L]$ with basis $e^{\gamma}$,\
$\gamma \in L$, and the following twisted group algebra product:
$$e^{\alpha}e^{\beta} = \varepsilon (\alpha ,\beta)e^{\alpha +
\beta}. $$
Let $B = \C[x] \otimes_\C \C[L]$ be the tensor
product of algebras.  Then the $n$-component boson-fermion
correspondence is the vector space isomorphism
$\sigma :F \to B, $
given by
\begin{equation}
\label{bf}
\sigma (\alpha^{(i_{1})}_{-m_{1}} \ldots
\alpha^{(i_{s})}_{-m_{s}}|k_{1},\ldots ,k_{n}\rangle ) = m_{1} \ldots
m_{s}x^{(i_{1})}_{m_{1}} \ldots x^{(i_{s})}_{m_{s}} \otimes
e^{k_{1}\delta_{1} + \ldots + k_{n}\delta_{n}} .
\end{equation}
Note that this is also equal to
\[
\langle 0|Q_n^{-k_n}\cdots Q_1^{-k_1}
\prod_{j=1}^n \Gamma_+^{(j)}(x^{(j)})\alpha^{(i_{1})}_{-m_{1}} \ldots
\alpha^{(i_{s})}_{-m_{s}}|k_{1},\ldots ,k_{n}\rangle \otimes e^{k_{1}\delta_{1} + \ldots + k_{n}\delta_{n}}
\]

%%%%%%%%%%%%%%%%%%
%%%%%%%%%%%%%%%%%%
%%%%%%%%%%%%%%%%%%

\section{Y.-P. Lee vector fields and the KP wave functions}

The goal of this section is to rewrite the formulas for the tangent action
of the Givental group restricted to genus $0$ in terms of wave functions of the multi-component KP hierarchy.

\subsection{The Givental group and the Y.-P. Lee vector fields.}

In \cite{G1,G2,G3,G4} (see also~\cite{FSZ,Lee2,LV,S} for further explanations) Givental introduced the action of the twisted loop group
on the space of formal Gromov-Witten theories. One of the possible ways of thinking about it is to say
that we consider an action of some special differential operators (depending on an element of twisted loop group) on
formal power series in variables $\hbar$ and $t_d^\mu, d\ge 0, \mu=1,\dots,n,$ of the type
\[
Z=\exp\left(\sum_{g\geq 0} \hbar^{g-1} F_g\right),
\]
where
\[
F_g=\sum_{k\ge 0} \sum_{\substack{d_1,\dots,d_k\ge 0 \\ 1\le \mu_1,\dots,\mu_k\le n}}
\bra \tau_{d_1}^{\mu_1}\dots \tau_{d_k}^{\mu_k}  \ket_g \frac{t_{d_1}^{\mu_1}\dots t_{d_k}^{\mu_k}}{k!},
\]
and the action is well-defined under some conditions imposed on $F_g$, $g\geq 0$.
Usually, $Z$ is called partition function, $F_g$ is called genus $g$ Gromov-Witten potential, and
the numbers $\bra \tau_{d_1}^{\mu_1}\dots \tau_{d_k}^{\mu_k}  \ket_g$ are called
correlators.

Let $G_\pm$ be positive and negative twisted loop groups of $GL_n$,
i.~e.,
\begin{gather*}
G_+=\{R(\zeta)=\Id + \sum_{i>0} R_i \zeta^{i}:\ R(-\zeta)^t R(\zeta)=\Id\},\\
G_-=\{S(\zeta)=\Id + \sum_{i>0} S_i \zeta^{-i}:\ S(-\zeta)^t S(\zeta)=\Id\}.
\end{gather*}
The corresponding Lie algebras are defined as follows:
\begin{gather*}
\g_+=\{r(\zeta)=\sum_{i>0} r_i \zeta^{i}:\ r(-\zeta)^t + r(\zeta)=0\},\\
\g_-=\{s(\zeta)=\sum_{i>0} s_i \zeta^{-i}:\ s(-\zeta)^t + s(\zeta)=0\}.
\end{gather*}
Givental defined actions of $G_\pm$ on the space of partition functions.
His definition goes through the quantization of certain Hamiltonians.
In \cite{Lee1,Lee2} Y.-P.~Lee computed the vector fields that correspond
to the Lie algebras $\g_\pm$ in Givental's action.

In order to present the Y.-P.~Lee formulas we need some further notations.
Let $(a)^\mu_\nu$ be the entries of the matrix $a$. We set
$(a)_{\mu\nu}=(a)^{\nu\mu}=(a)^\mu_\nu$, and we also use the notations
$(a)_\un^\mu=\sum_\nu (a)^\mu_\nu$, $(a)_\un^\un=\sum_{\mu,\nu} (a)^\mu_\nu$.

Let $s=\sum_{l\ge 1} s_l \zeta^{-l}$ be some element from $\g_-$. Then the corresponding
vector field $\widehat s$ is given by
the first order differential operator
\begin{multline}\label{eq:hat-s}
\widehat{s} = -\frac1{2\hbar} (s_3)_{\un,\un} - \sum_\mu (s_1)_{\un}^\mu \frac{\pa}{\pa t_0^\mu}
+ \frac1\hbar \sum_{d,\mu} (s_{d+2})_{\un,\mu} \, t_d^\mu\\
+ \sum_{
\substack{d,l\\ \mu, \nu}
}
(s_l)_\nu^\mu \, t_{d+l}^\nu \frac{\pa}{\pa t_d^\mu}
+ \frac1{2 \hbar} \sum_{
\substack{d_1,d_2 \\ \mu_2,\mu_2}
}
(-1)^{d_1} (s_{d_1+d_2+1})_{\mu_1,\mu_2} \, t_{d_1}^{\mu_1} t_{d_2}^{\mu_2}.
\end{multline}

Let $r=\sum_{l\ge 1} r_l \zeta^{l}$ be some element from $\g_+$. Then the corresponding
vector field $\widehat r$ is given by the second order differential operator:
\begin{multline}\label{eq:hat-r}
\widehat{r} =
- \sum_{
\substack{l \geq 1 \\ \mu}
}
(r_l)_{\un}^\mu \frac{\pa}{\pa t_{l+1}^\mu}\\
+\sum_{
\substack{d \geq 0, l \geq 1\\ \mu,\nu}
} \!\!
(r_l)_\nu^\mu \, t_d^\nu \frac{\pa}{\pa t_{d+l}^\mu}
+\frac{\hbar}2 \sum_{
\substack{d_1,d_2 \geq 0 \\ \mu_1,\mu_2}
} \!
(-1)^{d_1+1} (r_{d_1+d_2+1})^{\mu_1 \mu_2}
\frac{\pa^2}{\pa t_{d_1}^{\mu_1} \pa t_{d_2}^{\mu_2}}.
\end{multline}

\begin{rem}
Our conventions are a bit different from the standard ones in Givental's theory. That is,
throughout the paper we fix the unit $\un$ of the Frobenius manifold to be the sum of
the flat coordinates $t_0^\mu$ and the metric  to be the identity matrix. This reflects the formulas
for $\widehat s$ and $\widehat r$ above. The reason for these replacements will be clear after we discuss the van de
Leur approach to Frobenius structures via the $n$-component KP wave functions.
\end{rem}

%%%%%%%%%%%%%%%%%
%%%%%%%%%%%%%%%%%
%%%%%%%%%%%%%%%%%

\subsection{Genus $0$ reduction of Y.-P.~Lee formulas}

In this paper we are only interested in the genus zero part of the formal Gromov-Witten theories.
Following Givental, we call the genus zero part $F_0$ a formal Frobenius structure.
(To be more precise, a formal Frobenius structure is a solution of the following
system of PDEs: the string equation and the topological
recursion relations, see \cite{G3}).
Let $F$ be the series obtained from $F_0$ by putting
$t_d^\mu=0$ for $d>0$. Then $F$ defines the family of algebras (depending on
the point $t_0$) via the structure constants:
\[
c_{\mu_1,\mu_2}^{\mu_3}=\sum_\nu \frac{\pa^3 F}{\pa t_0^{\mu_1}\pa t_0^{\mu_2}\pa t_0^{\nu}}\eta^{\nu,\mu_3},
\]
where $\eta$ is a non-degenerate symmetric bilinear form. The
associativity of these algebras is equivalent to the celebrated
WDVV equation (see \cite{DVV,D1,D,W2}). In \cite{G3} Givental
defined the action of the twisted loop group on the space of
Frobenius structures. The most convenient way to describe this
action is via the genus zero restriction of the Y.-P.~Lee formulas given above.

Recall the function
\[
F(t_0^1,\dots, t_0^n)=\sum_{k\ge 0} \sum_{\mu_1,\dots,\mu_k}
\bra \tau_0^{\mu_1}\dots \tau_0^{\mu_k}  \ket_0 \frac{t_0^{\mu_1}\dots t_0^{\mu_k}}{k!}.
\]
We are now going  to extract the restriction of the Y.-P. Lee vector fields to $F$.
More precisely, let $Z=\sum_{g\ge 0} \hbar^{g-1} F_g$ be some partition function
and $g(\ve)$ be a family of elements of $G_\pm$, $g(0)=\Id$. Then
\[
g(\ve) Z=\exp(\sum_{g\ge 0} \hbar^{g-1} F_g(\ve)).
\]
Assume that the  tangent vector to $g(\ve)$ at $\ve=0$ is equal to $s$ (respectively, $r$).
Our goal is to express $\frac{\pa F(\ve)}{\pa\ve}$ at $\ve=0$ in terms of
$Z$. We denote these derivatives by $s.F$ (respectively, $r.F$).

We will need one more piece of notations. Namely, let
\begin{align}
\theta^{(0)}_\mu(t^1,\dots,t^n) & =1, \\
\label{t}
\theta^{(1)}_\mu(t^1,\dots,t^n) & =t^\mu,
\end{align}
and for $d\ge 2$ let
\begin{equation}\label{theta}
\theta_\mu^{(d)}(t^1_0,\dots,t^n_0) =\left. \frac{\pa F_0}{\partial t^{\mu}_{d-2}}\right |_{t_k^\nu=0, k>0}
  = \sum_{k\ge 0} \sum_{\mu_1,\dots,\mu_k}
\bra \tau_{d-2}^\mu \tau_0^{\mu_1}\dots \tau_0^{\mu_k}  \ket_0 \frac{t_0^{\mu_1}\dots t_0^{\mu_k}}{k!}.
\end{equation}
We also use the vector notation
\begin{equation}
\theta^{(d)}=(\theta_1^{(d)},\dots \theta_n^{(d)}).
\end{equation}
Say, $\theta^{(0)}=(1,1,\dots,1)$.

\begin{prop}\label{s.F}
Let $s = \sum_{l \geq 1} s_l \zeta^{-l}$. Then
\[
s.F=-\frac{1}{2}\theta^{(0)}s_3\theta^{(0)t}
+ \theta^{(0)} s_2 \theta^{(1)t}
-\theta^{(0)} s_1 \theta^{(2)t}
+
\frac{1}{2}\theta^{(1)} s_1\theta^{(1)t},
\]
where the upper index $t$ denotes the transposed vector or matrix.
\end{prop}
\begin{proof}
Since we are interested only in coefficients of monomials with $t_0^\mu$, formula $\eqref{eq:hat-s}$ turns into
$$
\widehat{s} = -\frac1{2\hbar} (s_3)_{\un,\un} - \sum_\mu (s_1)_{\un}^\mu \frac{\pa}{\pa t_0^\mu}
+ \frac1\hbar \sum_{\mu} (s_{2})_{\un,\mu} \, t_0^\mu
+ \frac1{2 \hbar} \sum_{\mu_2,\mu_2}
(s_{1})_{\mu_1,\mu_2} \, t_{0}^{\mu_1} t_{0}^{\mu_2}.
$$
The derivative of the the constant term is equal to $-\frac{1}{2}(s_3)_{\un,\un}$. Since $\un$ is the sum of coordinates in our case, we replace that by
$-\frac{1}{2} (1,\dots,1) s_3 (1,\dots,1)^t$.

Look at the $s_2$-term of this formula. It affects only degree $1$ terms in flat coordinates adding to them
$(s_{2})_{\un,i} \, t_i = (1,\dots,1) s_2 \theta^{(1)t}$.

Now look at the $s_1$-part of the above formula. The derivative
$- \sum_\mu (s_1)_\un^\mu \frac{\pa}{\pa t_0^\mu}$
just adds
$-(1,\dots,1) s_1 \theta^{(2)t}$ to $F$. The term
$\frac1{2 \hbar} \sum_{
 \mu_2,\mu_2}
(s_{1})_{\mu_1,\mu_2} \, t_{0}^{\mu_1} t_{0}^{\mu_2}$ can be written in terms of $\theta$-s as $
\frac{1}{2} \theta^{(1)} s_1 \theta^{(1)t}$.
The Proposition follows.
\end{proof}

\begin{prop}\label{r.F}
Let $r = \sum_{l\geq 1}r_l \zeta^{l}$. Then $r.F$ is equal to
\[
\sum_{l=1}^\infty\left(
-\theta^{(l+3)} r_l \theta^{(0)t} + \theta^{(l+2)} r_l \theta^{(1)t} +\frac{1}{2}
\sum_{m+m'=l-1} (-1)^{m'+1} \theta^{(m+2)} r_l\theta^{(m'+2)^t}\right).
\]
\end{prop}
\begin{proof}
Formula~\eqref{eq:hat-r} implies the following expression for the derivative of a particular coefficient (see~\cite{Lee2,FSZ}):
\begin{multline*}
r. \bra \tau_{d_1}^{\mu_1} \dots \tau_{d_n}^{\mu_n} \ket_0  = \\
- \sum_{l=1}^{\infty}
(r_l)_{\mu,\un}\bra \tau_{l+1}^\mu\tau_{d_1}^{\mu_1} \dots \tau_{d_n}^{\mu_n}\ket_0
 + \sum_{l=1}^{\infty} \sum_{i=1}^n
(r_l)_{\nu,\mu_i}
\bra
\tau_{d_1}^{\mu_1} \dots \tau_{d_i+l}^{\nu} \dots \tau_{d_n}^{\mu_n}
\ket_0\\
+\frac12 \sum_{l=1}^{\infty} \sum_{m+m'=l-1} (-1)^{m+1}
\sum_{I \sqcup J = \{1, \dots, n \}}
(r_l)^{\mu\nu}
\bra\tau_{m}^{\mu} \prod_{i \in I} \tau_{d_i}^{\mu_i}
\ket_0
\bra\tau_{m'}^{\nu} \prod_{i \in J} \tau_{d_i}^{\mu_i}
\ket_0.
\end{multline*}

We are interested only in correlators with $d_i=0$ for all $i$. Then the terms of this formula are collected into the derivative
of $F$ as follows. The generating function of the terms
$- \sum_{l=1}^{\infty} (r_l)_{\mu,\un}\bra \tau_{l+1}^\mu\tau_{0}^{\mu_1} \dots \tau_{0}^{\mu_n}\ket_0 $ in our language turns into
$-\sum_{l=1}^\infty \theta^{(l+3)} r_l (1 \dots 1)^t$. The generating function of the terms
$\sum_{l=1}^{\infty} \sum_{i=1}^n
(r_l)_{\nu,\mu_i}
\bra
\tau_{0}^{\mu_1} \dots \tau_{l}^{\nu} \dots \tau_{0}^{\mu_n}
\ket_0
$
turns into
$
\sum_{l=1}^\infty \sum_{\mu,\nu} \theta^{(l+2)}_\nu (r_l)_{\nu,\mu} t_\mu.
$
Finally, the term
$$
\frac12 \sum_{l=1}^{\infty} \sum_{m+m'=l-1} (-1)^{m+1}
\sum_{I \sqcup J = \{1, \dots, n \}}
(r_l)^{\mu\nu}
\bra\tau_{m}^{\mu} \prod_{i \in I} \tau_{0}^{\mu_i}
\ket_0
\bra\tau_{m'}^{\nu} \prod_{i \in J} \tau_{0}^{\mu_i}
\ket_0
$$
turns into
$
\frac{1}{2} \sum_{l=1}^\infty \sum_{m+m'=l-1} (-1)^{m'+1}
\theta^{(m+2)} r_l \theta^{(m'+2)t}.
$
Here we use that $(r_l)^{\mu\nu}=(r_l)_{\nu\mu}=(-1)^{l-1}(r_l)_{\mu\nu}$
\end{proof}

\subsection{$n$-KP hierarchies and Frobenius manifolds.}
In this subsection we review main definitions and results from \cite{vdL}.

We first recall the definition of the tau-functions.
Let $\bra 0|\in F^*$ be the left vacuum.
The $n$-component
tau-functions $\tau_\alpha(x)$ attached to some operator $A$ are defined as follows
\be
\tau(x)=\sum_{\alpha\in L} \tau_\alpha(x) e^{\alpha}= \sigma(A|0\ket),
\en
where $\sigma$ is defined by \eqref{bf}.
One can show that
\be
\tau_{k_1\delta_1+\dots + k_n\delta_n}(x)=\
\bra 0|Q_n^{-k_n}\dots Q_1^{-k_1} \prod_{i=1}^n \Gamma^{(i)}_+(x^{(i)}) A|0\ket.
\en
We note that if $A\in GL_\infty$ then the tau functions are polynomials
in variables $x^{(i)}_k$
and $\tau_{k_1\delta_1+\dots + k_n\delta_n}(x)=0$ if $k_1+\dots + k_n\ne 0$.
However, for $A\in\ov{GL}_\infty^+$ the tau-functions are infinite
power series (see the end of Section~\ref{semi}).

The wave function for $A=A(t)$ is equal to:
\[
\Phi_{ik}^\pm(A,x,z)=\sum_{\ell\in\Z}\Phi_{ik}^\pm(A,x)_\ell z^\ell=
\frac{\langle 0|\Gamma_+(x) Q_i^{\mp 1}\psi^{\pm(k)}(z)A|0\rangle}
{\langle 0|\Gamma_+(x) A|0\rangle}\, ,
\]
where
\[
\Gamma_+(x)=
\prod_{i=1}^n \Gamma_+^{(i)}(x^{(i)})\, .
\]
(We note that we are actually using the special type of wave functions from \cite{KvdL}  with
the parameter $\alpha=0$).
Note that the denominator $\langle 0|\Gamma_+(x) A|0\rangle$ is the tau-function
$\tau_0$ attached to $A$ and the numerator is obtained from
$\tau_{\pm(\delta_i-\delta_k)}$ by applying the bosonic (vertex operator) form of the
series $\psi^{\pm (k)} (z)$.
We put
\[
\Psi_{ik}^\pm(A,x,z)=\sum_{\ell=0}^\infty \Psi_{ik}^\pm(A,x)_\ell z^\ell=
\frac{\langle 0|\Gamma_+(x) Q_i^{\mp 1}A\psi^{\pm(k)}(z)|0\rangle}
{\langle 0|\Gamma_+(x) A|0\rangle}\,
\]
and the corresponding matrix
\[
\Psi^\pm(A,x,z)=
\Phi^\pm(A,x,z)A(z)\, .
\]
In what follows we sometimes drop the element $A$ or variables $x$ (or both)
and denote the corresponding functions $\Psi(A,x,z):=\Psi^+(A,x,z)$ simply by $\Psi(x,z)$ (or
simply by $\Psi$).

The following property of $\Psi$ (see \cite{vdL}) is crucial for us.
\begin{prop}
If $A(t)A(-t)^t=\Id$, then
\[
\Psi(A,x,z)\Psi(A,x,-z)^t=\Id
\]
after substitution of $x_{2k}^{(i)}=0$ for all $k=1,2,\ldots$, $1\le i\le n$.
\end{prop}

We keep the assumption $A(t)A(-t)^t=\Id$ for $A$ throughout the rest of this section.

We now define the KP version of the functions $\theta^{(k)}_i$ (see \eqref{t}, \eqref{theta}).
Slightly abusing notations, we denote these KP $\theta$ functions by the same symbols.
The identification will be explained below.
Set
\[
\theta(x_1,z)=(\theta_1(x_1,z),\dots,\theta_n(x_1,z))= \left. (1,\dots,1)\Psi(x,0)^t\Psi(x,z)\right |_{x_{k}^{(i)}=0,\ k\geq 2}.
\]
(By $x_1$ we denote the set of variables $x_1^{(1)},\dots,x_1^{(n)}$.)
Consider the decomposition
\[
\theta_i(x_1,z)=1+t^i(x_1) z+\sum_{d=2}^\infty \theta_i^{(d)}(x_1) z^d.
\]
Then
\begin{equation}\label{thetaPsi}
\theta^{(d)}=\left. (1,\dots,1)\Psi_0^t\Psi_d\right |_{x_{k}^{(i)}=0,\ k\geq 2}.
\end{equation}
The functions $t^i(x)$ play a role of the flat coordinates and the following
theorem constructs the corresponding Frobenius manifold (\cite{AKV}, \cite{vdL}):

\begin{thm}\label{Frobenius}
The function
\[
F=\frac{1}{2}\sum_{i=1}^n (t^i \theta^{(2)}_i-\theta^{(3)}_i)
\]
satisfies the WDVV equation in variables $t^i$.
\end{thm}

In what follows we will need the following properties of the functions $\theta_i$
and $F$. These properties are derived in \cite{vdL} via the analysis of the $n$-KP
hierarchies.

\begin{prop}
The following equalities hold:
\begin{gather}
\frac{\pa F}{\pa t^m}=\theta^{(2)}_m;\\
\label{TRR}
\frac{\pa^2\theta^{(s)}}{\pa t^k\pa t^l}=
\sum_{m=1}^n \frac{\pa^3 F}{\pa t^k\pa t^l\pa t^m}\frac{\pa\theta^{(s-1)}}{\pa t^m}.
\end{gather}
\end{prop}

\begin{rem}
Recall the definition of $\theta$ coming from the Gromov-Witten theories.
 Then it is easy to see that the equations \eqref{TRR} are contained in the set of
the topological recursion relations for $F_0$ (see \cite{G3}).
\end{rem}

We now rewrite the formulas for the vector fields from $\g_\pm$
in terms of $\Psi_k$. For a matrix $A$ we denote by $[A]$ the sum of all entries of $A$:
\[
[A]=(1,\dots,1)A (1,\dots,1)^t.
\]

\begin{prop}\label{.Psi}
\begin{multline}
s.F=\frac{1}{2}\left[\Psi_0^t\left(
-\Psi_2s_1\Psi_0^t
+\Psi_1s_1\Psi_1^t
-\Psi_0s_1\Psi_2^t
\right.\right. \\ \left.\left.
-\Psi_1s_2\Psi_0^t
+\Psi_0s_2\Psi_1^t
-\Psi_0s_3\Psi_0^t
\right)\Psi_0\right],
\end{multline}
and for any $l\geq 1$,
\begin{equation}
(r_l\zeta^{l}).F=
- \frac{1}{2}\left[\Psi_0^t\left( \sum_{i=0}^{l+3} (-1)^i \Psi_{l+3-i}r_l\Psi_i^t  \right)\Psi_0\right].
\end{equation}
\end{prop}

\begin{proof}
Follows from Propositions \ref{s.F} and \ref{r.F} and Formula \eqref{thetaPsi}. We also use here that $(s_l)_{\mu\nu}=(-1)^{l-1}(s_l)_{\nu\mu}$ and the same for matrices $r_l$, $l=1,2,\dots$.
\end{proof}

\section{Derivative formulas in the KP picture}
We start with the definition of the structure of $G_\pm$ modules
on $F$ (or rather on its completion $\bar F$).

Recall that we have fixed the basis $v^{(i)}_k$, $1\le i\le n$, $k\in\Z+\frac{1}{2}$ of the space $V$.
Let $\C^n$ be a $n$-dimensional vector space with a basis $e_i$, $i=1,\dots,n$.
Consider the isomorphism
\[
V\to \C^n\T\C[t,t,^{-1}], \quad v^{(i)}_k\mapsto e_i\T t^{-k-\frac{1}{2}}.
\]
We define the action of the groups $G_\pm$ and the algebras $\g_\pm$ on
$\C^n\T\C[t,t,^{-1}]$ via the map $\zeta^i\mapsto t^{-i}$ (see Remark \ref{pm}).
For instance, the element
$$
R(\zeta)=\Id + \sum_{i>0} R_i\zeta^i\in G_+
$$ acts as the operator $R(t)=\Id + \sum_{i>0} R_i\T t^{-i}$
on the completed space $V$ with respect to the grading defined by the powers of $t$
(this is an analogue of the energy grading).
This gives the homomorphisms
\[
G_\pm\to \ov{GL}_\infty^\pm,\quad \g_\pm\to\bar\gl_\infty^\pm
\]
and thus produces the
structure of $G_-$ ($\g_-$) module on $F=\Lambda^{\infty/2} V$ and
the structure of $G_+$ ($\g_+$) module on the completed space $\bar F$.
For instance, we can define
the tau-functions, corresponding to any element $S(\zeta)\in G_-$ and $R(\zeta)\in G_+$
(in the latter case these tau-functions are the power series in $x$).

\begin{rem}\label{pm}
We have fixed the action of $\zeta^i$ as the multiplication by
$t^{-i}$ in order to match two different styles of notations: one
coming from the Frobenius structures side (see \cite{G1},
\cite{FSZ}) and the other coming from the KP side (see
\cite{KvdL,KvdL2,vdL}). In fact, the group $G_+$, containing the
power series in variable $\zeta$, corresponds to the group of
operators, which are infinite series in $t^{-1}$ in the KP picture
and vice versa (see the rest of this section). In particular, the
group $G_+$ (which is called the upper-triangular group in
\cite{G1}, \cite{FSZ}) is embedded into the lower-triangular group
$\ov{GL}_\infty^+$.
\end{rem}

Let $A=A(t)$ be an element from the twisted loop group and $a$ be an element
from the twisted loop Lie algebra. Our goal in this section is to compute
the derivatives
\begin{equation}
\frac{\pa}{\pa\ve} \Psi(A\exp(\ve a),x,z)|_{\ve=0}
\end{equation}
We will consider two cases: $a=s(t)=\sum_{l>0} s_l t^{l}$ and $a=r_lt^{-l}$.
The first case is simple and the second is much more involved.

Recall  that the wave function for $A=A(t)$ is equal to:
\begin{equation}
 \label{wave01}
\Phi_{ik}^\pm(A,x,z)=\sum_{\ell\in\Z}\Phi_{ik}^\pm(A,x)_\ell z^\ell=
\frac{\langle 0|\Gamma_+(x) Q_i^{\mp 1}\psi^{\pm(k)}(z)A|0\rangle}
{\langle 0|\Gamma_+(x) A|0\rangle}\, ,
\end{equation}
where
\[
\Gamma_+(x)=
\prod_{i=1}^n \Gamma_+^{(i)}(x^{(i)})\, .
\]
The $\Psi=\Psi^+$ function is given by
\begin{equation}
\label{wave02}
\Psi_{ik}^\pm(A,x,z)=\sum_{\ell=0}^\infty \Psi_{ik}^\pm(A,x)_\ell z^\ell=
\frac{\langle 0|\Gamma_+(x) Q_i^{\mp 1}A\psi^{\pm(k)}(z)|0\rangle}
{\langle 0|\Gamma_+(x) A|0\rangle}\,
\end{equation}
and the corresponding matrix is given by
\begin{equation}
\label{wave03}
\Psi^\pm(A(t),x,z)=
\Phi^\pm(A(t),x,z)A(z)\, .
\end{equation}

For $f(z)=\sum_{i\in\Z} f_iz^i$ let $f(z)_+=\sum_{i\ge 0} f_iz^i$ and $f(z)_-=f(z)-f(z)_+$.
\begin{prop}\label{sKP}
Let $s=s(\zeta)=\sum_{i>0} s_i\zeta^{-i}\in\g_-$, $S_\ve=\exp(\ve s)\in G_-$. Then
\begin{equation}
\label{spsi}
\left.\frac{\pa}{\pa\ve} \Psi(AS_\ve,x,z)\right |_{\ve=0} =\Psi(A,x,z)s(z^{-1})
\end{equation}
or, equivalently,
\[
s.\Psi_k=\sum_{i=0}^{k-1} \Psi_is_{k-i}.
\]
\end{prop}
\begin{proof}
Direct computation.
\end{proof}

\begin{thm}\label{rKP}
Let $ r\zeta^{\ell}\in \g_+$, $R_\ve=\exp(\ve r\zeta^{\ell})\in G_+$. Then
\begin{multline}\label{rpsi}
\left.\frac{\pa}{\pa\ve} \Psi(AR_\ve,x,z)\right |_{\ve=0}
=\Psi(A,x,z)rz^{-\ell}\\
- \sum_{p=1}^\ell \sum_{q=0}^{\ell-p} (-1)^{\ell-p-q}\Psi(A,x)_q r \Psi(A,x)_{\ell-p-q}^t \Psi(A,x,z){z^{-p}}
\end{multline}
or, equivalently,
\[
(r\zeta^\ell).\Psi_k=\Psi_{\ell+k}r - \sum_{p=1}^\ell\sum_{q=0}^{\ell-p}(-1)^{\ell-p-q}\Psi_q r \Psi_{\ell-p-q}^t\Psi_{p+k}.
\]
\end{thm}

\begin{rem}
Formulas \eqref{spsi} and \eqref{rpsi} can be written in a uniform way:
\begin{multline*}
\left.\frac{\pa}{\pa\ve} \Psi(AG_\ve,x,z) \right |_{\ve=0}\\ =
\Psi(A,x,z)g(z^{-1})
-\left(\Psi(A,x,z)g(z^{-1})\Psi(A,x,-z)^t\right)_- \Psi(A,x,z),
\end{multline*}
where $g(\zeta) + g(-\zeta)^t=0$ and $G_\ve=\exp(\ve g(\zeta))$.
This is simple for $g(\zeta)=s(\zeta)$ and will be proved below
for $g(\zeta)=r\zeta^\ell$.
\end{rem}

The proof of Theorem \ref{rKP}  occupies the rest of this section.
We start with simple Lemma.
\begin{lem}
The following equality holds
up to $O(\epsilon^2)$
\begin{multline}
\label{Psi}
\Psi^+(AR_\ve,x,z)-\Psi^+(A,x,z)
=\epsilon\Psi^+(A,x,z)rz^{-\ell} \\
\qquad\qquad\qquad\qquad
+ \epsilon\left(\frac{\langle 0|\Gamma_+(x) Q_i^{- 1}
\psi^{+(k)}(z)Art^{-\ell}|0\rangle}{\langle 0|\Gamma_+(x) A|0\rangle}\right)_{1\le i,k,\le n}\cdot A(z) \\
-\epsilon\frac{\langle 0|\Gamma_+(x) Art^{-\ell}|0\rangle}{\langle 0|\Gamma_+(x) A|0\rangle}\Psi^+(A,x,z)
\end{multline}
\end{lem}
\begin{proof}
A straightforward calculation using the formula's (\ref{wave01}-\ref{wave03}).
\end{proof}
In order to compute the right hand side of \eqref{Psi} we introduce the generating series
\[
r(w,y)=\sum_{a,b=1}^n r_{ab}:\psi^{+(a)}(w)\psi^{-(b)}(y):
\]
and replace $rt^{-l}$ in \eqref{Psi} with $r(w,y)$.
Later on we will take the residue of $y^{-\ell} \lim_{w\to y} r(w,y)$
which is equal to $r t^{-\ell}$.

%Observe that
%up to $O(\epsilon^2)$ one has
%\begin{equation}
%\label{wave}
%\begin{split}
%\Phi^+_{ik}&(A(1+\epsilon r(w,y)),x,z)=
%\frac{\langle 0|\Gamma_+(x) Q_i^{- 1}\psi_k^{+(k)}(z)A(1+\epsilon r(w,y))|0\rangle}
%{\langle 0|\Gamma_+(x) A(1+\epsilon r(w,y))|0\rangle}\\[2mm]
%=&
%\frac{\langle 0|\Gamma_+(x) Q_i^{- 1}\psi_k^{+(k)}(z)A|0\rangle+\epsilon \langle 0|
%\Gamma_+(x) Q_i^{- 1}\psi_k^{+(k)}(z)Ar(w,y)|0\rangle}
%{\langle 0|\Gamma_+(x) A|0\rangle+\epsilon\langle 0|\Gamma_+(x) Ar(w,y) |0\rangle}\\[2mm]
%=&
%\frac{\langle 0|\Gamma_+(x) Q_i^{- 1}\psi^{+(k)}(z)A|0\rangle+\epsilon \langle 0|\Gamma_+(x)
%Q_i^{- 1}\psi^{+(k)}(z)Ar(w,y)|0\rangle}
%{\langle 0|\Gamma_+(x) A|0\rangle
%\left(1+\epsilon\frac{\langle 0|\Gamma_+(x) A r(w,y)|0\rangle}{\langle 0|\Gamma_+(x) A|0\rangle}\right)}\\[2mm]
%=&\left(\frac{\langle 0|\Gamma_+(x) Q_i^{- 1}\psi_k^{+(k)}(z)A|0\rangle}
%{\langle 0|\Gamma_+(x) A|0\rangle}+\epsilon\frac{\langle 0|\Gamma_+(x) Q_i^{- 1}\psi_k^{+(k)}(z)
%Ar(w,y)|0\rangle}{\langle 0|\Gamma_+(x) A|0\rangle}\right)\left(1-\epsilon\frac{\langle 0|\Gamma_+(x) A %r(w,y)|0\rangle}{\langle 0|\Gamma_+(x) A|0\rangle}\right)\\[2mm]
%=&\Phi^+_{ik}(A,x,z)+\epsilon\frac{\langle 0|\Gamma_+(x) Q_i^{- 1}\psi_k^{+(k)}(z)Ar(w,y)|0\rangle}{\langle 0|\Gamma_+(x)
% A|0\rangle}-\epsilon\frac{\langle 0|\Gamma_+(x) Ar(w,y) |0\rangle}{\langle 0|\Gamma_+(x)
%A|0\rangle}\Phi^+_{ik}(A,x,z)\, .
%\end{split}
%\end{equation}
%Recall that
%$$
%\Psi^+(A(1+\epsilon r t^{-\ell}),x,z)=\Phi^+(A(1+\epsilon rt^{-\ell}),x,z)A(1+\epsilon rz^{-\ell}).
%$$
%We will calculate the right hand side of \eqref{Psi}.

In what follows we will need the following lemma:
\begin{lem}
$a)$. $A\otimes A$ commutes with
\[
S=\sum_{k=1}^n \mbox{Res}_z\,\psi^{+(k)}(z)\otimes \psi^{-(k)}(z)\, ,
\]
i.e.,
\begin{multline}\label{S}
\sum_{k=1}^n \mbox{Res}_z\,\psi^{+(k)}(z)A|0\rangle\otimes \psi^{-(k)}(z)A|0\rangle\\
=\sum_{k=1}^n \mbox{Res}_z\, A\psi^{+(k)}(z)|0\rangle\otimes A\psi^{-(k)}(z)|0\rangle=0
\end{multline}
and
\[
S(|0\rangle\otimes |0\rangle)=0
\]
$b)$. We have
\begin{equation}
\begin{split}\label{Res}
\sum_{k=1}^n \mbox{Res}_z\,&\psi^{+(k)}(z):\psi^{+(a)}(w)\psi^{-(b)}(y):|0\rangle\otimes \psi^{-(k)}(z)|0\rangle\\[2mm]
&=-\mbox{Res}_z\,\delta(z-y)\psi^{+(a)}(w)|0\rangle\otimes\psi^{-(b)}(z)|0\rangle\, \\[2mm]
&=-\sum_{i\ge 0}\psi_{-i-\frac12}^{+(a)}w^i|0\rangle\otimes\sum_{j\ge 0}\psi_{-j-\frac12}^{-(b)}y^j|0\rangle\,
\end{split}
\end{equation}
\end{lem}

The central point of the proof is the following Proposition:

\begin{prop}\label{centraal}
\begin{multline}
\label{wave5}
\left(
\frac{\langle 0|Q_i^{-1}\Gamma_+(x)\psi^{+(k)}(z)A:\psi^{+(a)}(w)\psi^{-(b)}(y):|0\rangle}
{\langle 0|\Gamma_+(x) A|0\rangle}\right)_{1\le i,k,\le n}\\
\shoveleft{=\frac{\langle 0|\Gamma_+(x)A:\psi^{+(a)}(w)\psi^{-(b)}(y):|0\rangle}
{\langle 0|\Gamma_+(x) A|0\rangle}\Phi^+(A,x,z)}\\
-\sum_{p,q,s\ge 0}\Psi^+(A,x)_qw^q E_{ab}\Psi^-(A,x)_s \frac{y^{p+s}}{z^{p+1}}\Phi^+(A,x,z)
\end{multline}
\end{prop}

We first deduce Theorem \ref{rKP} from Proposition \ref{centraal}.\\
{\it Proof of Theorem \ref{rKP}}. Multiplying \eqref{wave5} by $r_{ab}$ and summing over all
$a$ and $b$, we obtain
\begin{multline}
\label{wave6}
\left(
\frac{\langle 0|Q_i^{-1}\Gamma_+(x)\psi^{+(k)}(z)Ar(w,y)|0\rangle}
{\langle 0|\Gamma_+(x) A|0\rangle}\right)_{1\le i,k,\le n}\\
\shoveleft{=\frac{\langle 0|\Gamma_+(x)Ar(w,y)|0\rangle}
{\langle 0|\Gamma_+(x) A|0\rangle}\Phi^+(A,x,z)}\\
-\sum_{p,q,s\ge 0}\Psi^+(A,x)_qw^q r \Psi^-(A,x)_s \frac{y^{p+s}}{z^{p+1}}\Phi^+(A,x,z)
\end{multline}
Now multiply from the right by $A(z)y^{-\ell}$, take $w=y$ and select the coefficient of $y^{-1}$. Using (\ref{wave01}
-\ref{wave03}) and
(\ref{Psi}) this gives up to $O(\epsilon^2)$:
\begin{multline}
\label{Psi2}
\Psi^+(AR_\ve,x,z)=\Psi^+(A,x,z)+\epsilon\Psi^+(A,x,z)rz^{-\ell}\\
- \sum_{p=1}^\ell \sum_{q=0}^{\ell-p} \Psi^+(A,x)_q r \Psi^-(A,x)_{\ell-p-q} \Psi^+(A,x,z){z^{-p}}
\end{multline}
Recall that with the restriction $x^{(i)}_{2\ell}=0$, one has
\[
\Psi^-(A,x,z)=\Psi^+(A,x,-z)^t,
\]
(see \cite{vdL}). Thus we obtain
\begin{multline}
\label{Psi3}
\Psi^+(AR_\ve,x,z)=\Psi^+(A,x,z)+\epsilon\Psi^+(A,x,z)rz^{-\ell}\\
- \ve \sum_{p=1}^\ell \sum_{q=0}^{\ell-p} (-)^{\ell-p-q}\Psi^+(A,x)_q r \Psi^+(A,x)_{\ell-p-q}^t \Psi^+(A,x,z){z^{-p}}
\end{multline}
Theorem is proved.\hfill$\square$
\\
\

In the rest of the section we prove Proposition \ref{centraal}.\\
{\it Proof of Proposition \ref{centraal}.} Our goal is to compute the expression
\begin{multline}
\label{T}
T(z)=T_{ij}(z)=\sum_{k=1}^n \Bigl(\frac{\langle 0|
Q_i^{-1}\Gamma_+(x)\psi^{+(k)}(z)A:\psi^{+(a)}(w)\psi^{-(b)}(y):|0\rangle}
{\langle 0|\Gamma_+(x) A|0\rangle}\\
\otimes  \frac{\langle 0|Q_j\Gamma_+(x')\psi^{-(k)}(z)A|0\rangle}
{\langle 0|\Gamma_+(x') A|0\rangle}\Bigr)
\end{multline}
If we put $x=x'$ in (\ref{T}) and forget the tensor product, then (\ref{T}) is the
multiplication of the $i$-th row of the left-hand side of (\ref{wave5})
with the $j$-th row of $\Phi^-(A,x,z)$. So we obtain the left-hand side of
(\ref{wave5}) by putting the results of the calculations of (\ref{T}) for all $1\le i,j\le n$ in a matrix, multiplying this by $\Phi^+(A,x,z)$ and using that $\Phi^+(A,x,z)=\Phi^-(A,x,z)^{t\, -1}$.

Let $T(z)=\sum_{k\in \Z} T_k z^k$ and $T(z)_+=\sum_{k\ge 0} T_kz^k$, $T(z)_-=T(z)-T(z)_+$.
We first compute $T(z)_-$.

Using \eqref{S} we obtain that $T_{-1}=\mbox{Res}_z T(z)$ is equal to
\begin{multline}
\label{wave2}
\sum_{k=1}^n \mbox{Res}_z\, \Bigl(\frac{\langle
0|Q_i^{-1}\Gamma_+(x)A\psi^{+(k)}(z):\psi^{+(a)}(w)\psi^{-(b)}(y):|0\rangle}
{\langle 0|\Gamma_+(x) A|0\rangle}\\
\otimes\frac{\langle 0|Q_j\Gamma_+(x')A\psi^{-(k)}(z)|0\rangle}
{\langle 0|\Gamma_+(x') A|0\rangle}\Bigr)
\end{multline}

Using \eqref{Res} we rewrite the last formula as
\begin{equation}\label{zy}
\begin{split}
&-\mbox{Res}_z\,\delta(z-y)\frac{\langle 0|Q_i^{-1}\Gamma_+(x)A\psi^{+(a)}(w)|0\rangle}{\langle 0|\Gamma_+(x) A|0\rangle}\otimes\frac{\langle 0|Q_j\Gamma_+(x')A\psi^{-(b)}(z)|0\rangle}{\langle 0|\Gamma_+(x') A|0\rangle}\\[2mm]
&=-\sum_{r\ge 0}\frac{\langle 0|Q_i^{-1}\Gamma_+(x)A\psi_{-r-\frac12}^{+(a)}w^r|0\rangle}
{\langle 0|\Gamma_+(x) A|0\rangle}
\otimes\sum_{s\ge 0}\frac{\langle 0|Q_j\Gamma_+(x')A\psi_{-s-\frac12}^{-(b)}y^s|0\rangle}{\langle 0|\Gamma_+(x')
A|0\rangle}\\[2mm]
&=-\sum_{r\ge 0}\Psi^+_{ia}(A,x)_rw^r\otimes\sum_{s\ge 0}\Psi^-_{jb}(A,x')_s y^s
\, .
\end{split}
\end{equation}
Now, the derivative of the second factor of the first line of \eqref{zy} with respect to
$-\sum_{i=1}^n\frac{\partial}{\partial {x'}^{(i)}_1}$, produces
 an extra term $-\sum_{i=1}^n \alpha_1^{(i)}$ between $\Gamma_+(x')$ and $A$.
We get two of such terms, one for the numerator and one for the denominator.
This element commutes with $A$ and gives $0$ when acting on the vacuum, hence the "denominator"-term is zero. Since
\[
-\sum_{i=1}^n [ \alpha_1^{(i)},\psi^{-(k)}_j  ]=\psi^{-(k)}_{j+1}\, ,
\]
one deduces that
\begin{multline}
T_{-2}=\sum_{k=1}^n \mbox{Res}_z\,z  \Bigl(\frac{\langle 0|Q_i^{-1}\Gamma_+(x)\psi^{+(k)}(z)A:\psi^{+(a)}(w)\psi^{-(b)}(y):|0\rangle}
{\langle 0|\Gamma_+(x) A|0\rangle}\\
\shoveright{ \otimes \frac{\langle 0|Q_j\Gamma_+(x')\psi^{-(k)}(z)A|0\rangle}
{\langle 0|\Gamma_+(x') A|0\rangle}\Bigr)}\\
=-y\sum_{r\ge 0}\Psi^+_{ia}(A,x)_rw^r\otimes\sum_{s\ge 0}\Psi^-_{jb}(A,x')_s y^s
\end{multline}
Repeating this procedure, we obtain
\begin{equation}
\label{wave3}
T(z)_-=-\sum_{p=0}^\infty \sum_{r\ge 0}\Psi^+_{ia}(A,x)_rw^r\otimes\sum_{s\ge 0}\Psi^-_{jb}(A,x')_s \frac{y^{p+s}}{z^{p+1}}
\end{equation}
We now want to calculate
the $+$-part of $T(z)$ (see (\ref{T})).
For this we
%
%
% I HAVE DELETED THE PART ABOUT THE TAU FUNCTION
% FIRST OF ALL I DO NOT THINK WE NEED THIS
% SECONDLY I DO NOT THINK IT IS TRUE
% IT IS TRUE WITHOUT THE NORMAL ORDERING.
%
%
% use that
% \[
% \Gamma_+(x)A:\psi^{+(a)}(w)\psi^{-(b)}(y):|0\rangle
% \]
% is also a tau function, i.e., it satisfies
% \[
% \begin{split}
% \mbox{Res}_z\,&
% \sum_{k=1}^n\langle 0|Q_i^{-1}\Gamma_+(x)\psi^{+(k)}(z)A:\psi^{+(a)}(w)\psi^{-(b)}(y):|0\rangle\\[2mm]
%&\qquad\qquad\otimes\langle 0|Q_j\Gamma_+(x')\psi^{-(k)}(z)A:\psi^{+(a)}(w)\psi^{-(b)}(y):|0\rangle=0\, ,
%\end{split}
%\]
%Now
look at
$
\langle 0|Q_i^{-1}\Gamma_+(x)\psi^{+(k)}(z)
$
and commute  $\Gamma_+(x)$ with $\psi^{+(k)}(z)$.  Using Theorem \ref{theorem} and (\ref{comgam}) we obtain
\begin{multline}
\langle 0|Q_i^{-1}\Gamma_+(x)\psi^{+(k)}(z)=
z^{\delta_{ik}-1}\gamma(x^{(k)}, [z])\langle 0|Q_i^{-1}Q_k\Gamma_+(x)\Gamma_{+}^{(k)}(-
[z^{-1}])\\
=z^{\delta_{ik}-1}\gamma(x^{(k)}, [z])\exp\left(-\sum_{j=1}^\infty
\frac{\partial}{\partial x_j^{(k)}}\frac{z^{-j}}{j}\right)\langle 0|Q_i^{-1}Q_k\Gamma_+(x)
\end{multline}
which means that
\begin{multline}
\left(
\frac{\langle 0|Q_i^{-1}\Gamma_+(x)\psi^{+(k)}(z)A:\psi^{+(a)}(w)\psi^{-(b)}(y):|0\rangle}{\langle 0|\Gamma_+(x)
A|0\rangle}\right)_{1\le i,k\le n}\\
=\sum_{p=0}^\infty B(x)_p z^{-p}\sum_{k=1}^n \gamma(x^{(k)},[z])E_{kk}\, ,
\end{multline}
with
\begin{multline}
\sum_{p=0}^\infty B_{ik}(x)_p z^{-p}=\quad\left(\langle 0|\Gamma_+(x) A|0\rangle\right)^{-1} z^{\delta_{ik}-1}\\
\times\exp\left(-\sum_{j=1}^\infty
\frac{\partial}{\partial x_j^{(k)}}\frac{z^{-j}}{j}\right)\langle 0|Q_i^{-1}Q_k\Gamma_+(x)A:\psi^{+(a)}(w)
\psi^{-(b)}(y):|0\rangle.
\end{multline}
In particular
\[
B_{ik}(x)_0=\delta_{ik}\frac{\langle 0|\Gamma_+(x)A:\psi^{+(a)}(w)\psi^{-(b)}(y):|0\rangle}{\langle 0|\Gamma_+(x)
A|0\rangle}
\]
In a similar way we have that
\begin{multline*}
\langle 0|Q_j\Gamma_+(x)\psi^{-(k)}(z)A|0\rangle\\
=z^{\delta_{jk}-1}\gamma(-x^{(k)}, [z])\langle 0|Q_jQ_k^{-1}\Gamma_+(x)\Gamma_{+}^{(k)}(
[z^{-1}])A| 0\rangle\\
=z^{\delta_{jk}-1}\gamma(-x^{(k)}, [z])\exp\left(\sum_{j=1}^\infty
\frac{\partial}{\partial x_j^{(k)}}\frac{z^{-j}}{j}\right)\langle 0|Q_jQ_k^{-1}\Gamma_+(x)A| 0\rangle
\end{multline*}
and thus
\[
\left(\frac{\langle 0|Q_j\Gamma_+(x')\psi^{-(k)}(z)A|0\rangle}
{\langle 0|\Gamma_+(x') A|0\rangle}\right)_{1\le j,k\le n}=
\sum_{p=0}^\infty C(x')_p z^{-p}\sum_{k=1}^n \gamma(-{x'}^{(k)},[z])E_{kk}\, ,
\]
with $C(x')_0=I_n$.

Now forget the tensor product in (\ref{T}). We see that this is as row times a column in a matrix multiplication.
Putting
$x=x'$, one has that for every $k$
\[
\gamma(x^{(k)},[z])
\gamma(-{x'}^{(k)},[z])=1\, .
\]
Thus the $+$-part of (\ref{T}) (with the tensor product replaced with the usual product)
is equal to $\sum_k B_{ik}(x)_0 C_{jk}(x)_0$, or in other words
\begin{multline}
\label{wave4}
\sum_{k=1}^n  \Bigl(
\frac{\langle 0|Q_i^{-1}\Gamma_+(x)\psi^{+(k)}(z)A:\psi^{+(a)}(w)\psi^{-(b)}(y):|0\rangle}
{\langle 0|\Gamma_+(x) A|0\rangle}\\
\times \frac{\langle 0|Q_j\Gamma_+(x)\psi^{-(k)}(z)A|0\rangle}
{\langle 0|\Gamma_+(x) A|0\rangle}
\Bigr)_+\\
=\frac{\langle 0|\Gamma_+(x)A:\psi^{+(a)}(w)\psi^{-(b)}(y):|0\rangle}
{\langle 0|\Gamma_+(x) A|0\rangle}\delta_{ij}.
\end{multline}
The Proposition \ref{centraal} follows.\hfill$\square$

\section{The main theorem}
In this section we compare the derivatives for the KP wave functions
(obtained in the previous section) with the restriction of the Y.-P.~Lee
derivatives written in flat coordinates.

We note that the Frobenius structure defined in Theorem \ref{Frobenius}
can be rewritten as
\[
F(x_1,\dots,x_n)=
\left.\frac{1}{2} \left[ -\Psi_0^t\Psi_3 + \Psi_0^t\Psi_2\Psi_1^t\Psi_0 \right]\right |_{x^{(i)}_k=0, k\ge 2}.
\]
Here $x_i=x^{(i)}_1$ and $[M]$ denotes the sum of elements of a matrix $M$,
i.~e., $[M]=(1,\dots,1)M(1,\dots,1)^t$.
%We note that this formula is written in terms of the coordinates $x_i$.
In Proposition \ref{flat} we derive the formula for the vector fields
derivatives of $F$ in the coordinates $t^i$. Namely, suppose
we have a family of the group elements $A(\ve;\zeta)$. Then our
general procedure defines the family of $\Psi_k(\ve)$, $\theta^{(k)}(\ve)$ and $F(\ve)$.
In particular, the flat coordinates $t^i(\ve)$ also depend on $\ve$.
Considering $F(\ve)$ as the family of functions in variables $x_i$,
we have the corresponding derivatives at $\ve=0$, which we denote by
$\Dot F$. For any $\ve$ we have an equality
\[
F(\ve;x_1,\dots,x_n)=F(\ve; t^1(\ve),\dots,t^n(\ve)).
\]
Our goal is to find the derivative of $F(\ve)$ with ``frozen'' variables
$t_i(\ve)$, i.e. we want to forget the dependence of the variables
$t_i(\ve)$ on $\ve$. We keep only the dependence on $\ve$ of the coefficients of
$F(\ve)$, written as a series in $t^i(\ve)$.
We denote the derivative in ``frozen'' coordinates by $\frac{\partial F}{\partial\epsilon}$.
\begin{prop}\label{flat}
The derivative in flat coordinates is given by the formula
\[
\frac{\partial F}{\partial\epsilon}= \frac{1}{2}\left[
\Psi_0^t(-\Dot\Psi_3\Psi_0^t + \Dot\Psi_2\Psi_1^t
-\Dot\Psi_1\Psi_2^t + \Dot\Psi_0\Psi_3^t) \Psi_0 \right]
\]
\end{prop}
\begin{proof}
We have
\begin{align*}
\frac{\partial F}{\partial\epsilon} & =  \Dot F - \sum_{i=1}^n \frac{\partial F}{\partial t_i} \Dot t_i \\
& = \Dot F - [\Psi_0^t \Psi_2 \left(\Psi_1^t \Psi_0\right)^\cdot] \\
& =  \frac{1}{2}\left[-\left(\Psi_0^t\Psi_3\right)^\cdot\Psi_0^t\Psi_0+\left(\Psi_0^t\Psi_2\right)^\cdot\Psi_1^t\Psi_0+
\Psi_0^t\Psi_2\left(\Psi_1^t\Psi_0\right)^\cdot\right] \\
& {\ } \qquad - [\Psi_0^t \Psi_2 \left(\Psi_1^t \Psi_0\right)^\cdot] \\
& = \frac{1}{2}\left[
-\left(\Psi_0^t\Psi_3\right)^\cdot\Psi_0^t\Psi_0
+\left(\Psi_0^t\Psi_2\right)^\cdot\Psi_1^t\Psi_0 \right. \\
& {\ }  \qquad
\left. -\left(\Psi_0^t\Psi_1\right)^\cdot\Psi_2^t\Psi_0
+\left(\Psi_0^t\Psi_0\right)^\cdot\Psi_3^t\Psi_0
\right] \\
& = \frac{1}{2}\left[
-\Psi_0^t\Dot\Psi_3\Psi_0^t\Psi_0
+\Psi_0^t\Dot\Psi_2\Psi_1^t\Psi_0
-\Psi_0^t\Dot\Psi_1\Psi_2^t\Psi_0
+\Psi_0^t\Dot\Psi_0\Psi_3^t\Psi_0
\right]
\end{align*}
Here the first equation is exactly the formula for the derivative corrected
with respect to the first order change of coordinates with respect to
$\epsilon$, and for the rest we used $\Psi(-z)^t\Psi(z)=\Id$ and the fact that $[A^t]=[A]$.
\end{proof}

We now prove that formulas from  Propositions \ref{.Psi} and
\ref{flat} do coincide (we use the expressions for the derivatives
$\Dot\Psi_i$ from Proposition \ref{sKP} and Theorem \ref{rKP}).

\begin{prop}
Let $s=s(\zeta)\in\g_-$. Then we have
\begin{multline}\label{Giv}
-(s.\Psi_3)\Psi_0^t + (s.\Psi_2)\Psi_1^t - (s.\Psi_1)\Psi_2^t + (s.\Psi_0)\Psi_3^t\\=
- \Psi_0s_1\Psi_2^t + \Psi_1s_1\Psi_1^t - \Psi_2s_1\Psi_0^t +
\Psi_0s_2\Psi_1^t + \Psi_1s_1\Psi_0^t - \Psi_0s_3\Psi_0^t.
\end{multline}
\end{prop}
\begin{proof}
Follows from Proposition  \ref{sKP}.
\end{proof}

\begin{prop}
Let $r\zeta^l\in\g_+$. Then we have
\begin{multline}
-(r\zeta^l.\Psi_3)\Psi_0^t + (r\zeta^l.\Psi_2)\Psi_1^t - (r\zeta^l.\Psi_1)\Psi_2^t + (r\zeta^l.\Psi_0)\Psi_3^t\\
= - \sum_{i=0}^{l+3} (-1)^i \Psi_{l+3-i} r \Psi^t_{l+3-i}.
\end{multline}
\end{prop}
\begin{proof}
From Theorem \ref{rKP} we obtain
\begin{gather*}
r\zeta^l.\Psi_0=\Psi_l r - \sum_{q=1}^l \sum_{p=0}^{l-q}
(-1)^{l-p-q}
\Psi_p r \Psi^t_{l-p-q} \Psi_q,\\
r\zeta^l.\Psi_1=\Psi_{l+1} r - \sum_{q=1}^l \sum_{p=0}^{l-q}
(-1)^{l-p-q}
\Psi_p r \Psi^t_{l-p-q} \Psi_{q+1},\\
r\zeta^l.\Psi_2=\Psi_{l+2} r - \sum_{q=1}^l \sum_{p=0}^{l-q}
(-1)^{l-p-q}
\Psi_p r \Psi^t_{l-p-q} \Psi_{q+2},\\
r\zeta^l.\Psi_3=\Psi_{l+3} r - \sum_{q=1}^l \sum_{p=0}^{l-q}
(-1)^{l-p-q} \Psi_p r \Psi^t_{l-p-q} \Psi_{q+3}.
\end{gather*}
Substituting these expressions we rewrite the left hand side of \eqref{Giv}
as
\begin{multline}\label{pq}
-\Psi_{l+3} r \Psi_0^t + \Psi_{l+2} r \Psi_1^t - \Psi_{l+1} r \Psi_2^t +
\Psi_{l} r \Psi_3^t\\ + \sum_{q=1}^l \sum_{p=0}^{l-q}
(-1)^{l-q+p} \Psi_p r \Psi_{l-p-q}^t
\left(\Psi_{q+3}\Psi_0^t - \Psi_{q+2}\Psi_1^t + \Psi_{q+1}\Psi_2^t - \Psi_{q}\Psi_3^t\right)
\end{multline}
The sum inside  brackets is equal to
\be
\sum_{s=1}^q (-1)^s \Psi_{q-s} \Psi_{s+3}^t.
\en
Therefore \eqref{pq} is equal to
\begin{multline*}
-\Psi_{l+3} r \Psi_0^t + \Psi_{l+2} r \Psi_1^t - \Psi_{l+1} r \Psi_2^t +
\Psi_{l} r \Psi_3^t\\ +
\sum_{p=0}^{l-1} \sum_{s=1}^l (-1)^{l-p+s} \Psi_p r
\left(\sum_{s\le q\le l-p} (-1)^q \Psi_{l-p-q}^t \Psi_{q-s} \right)
\Psi^t_{s+3}.
\end{multline*}
Since $\Psi^t(-z)\Psi(z)=\Id$, the sum in brackets vanishes unless $s=l-p$. In the latter case the sum
is equal to $(-1)^{l-p}$.
Proposition follows.
\end{proof}

Summarizing, we obtain the following theorem:
\begin{thm}
Let $a$ be an element from the twisted loop Lie algebra.
For any $A$ from the twisted loop group the effect of the derivative
\[
\left. \frac{\pa}{\pa\ve} \Psi(A\exp(\ve a),x,z)\right |_{\ve=0}
\]
on the corresponding Frobenius structure
is given by the genus zero no-descendants restriction of the Y.~-P.~Lee formulas.
\end{thm}

\section*{Acknowledgments}
The work of EF was partially supported
by the RFBR grants 09-01-00058, 07-02-00799, and NSh-3472.2008.2, by Pierre Deligne fund
based on his 2004 Balzan prize in mathematics and by Alexander von
Humboldt Fellowship. The work of JvdL was partially supported by the European Science
Foundation Program MISGAM and the Marie Curie RTN  ENIGMA. The work of SS
was partially supported by the Vidi grant of NWO.

\end{document}